\documentclass{amsart}

\usepackage{fullpage}
\usepackage{amsmath}
\usepackage{amsfonts}
\usepackage{amssymb}
\usepackage{yhmath}

\input xy
\xyoption{matrix}\xyoption{arrow}\xyoption{curve}  \xyoption{frame}

\DeclareMathOperator{\Hom}{Hom}

\DeclareMathOperator{\stHom}{\underline{\Hom}}
\DeclareMathOperator{\sthom}{\underline{\hom}}
\DeclareMathOperator{\Ext}{Ext}
\DeclareMathOperator{\ext}{ext}

\DeclareMathOperator{\Endo}{End}

\DeclareMathOperator{\stEndo}{{\underline{\Endo}}}

\DeclareMathOperator{\gldim}{{gl.dim}}

\DeclareMathOperator{\aut}{{Aut}}

\DeclareMathOperator{\add}{{add}}

\DeclareMathOperator{\ind}{{ind}}

\DeclareMathOperator{\silt}{{silt}}

\DeclareMathOperator{\ogr}{gr}
\mathchardef\mhyphen="2D

\DeclareMathOperator{\ormod}{mod}

\newcommand{\gr}[1][]{\ogr^{#1 \!} \mhyphen}

\newcommand{\stgr}[1][]{\underline{\ogr}^{#1 \!} \mhyphen}

\newcommand{\rmod}[1][]{\ormod \mhyphen #1}

\newcommand{\stmod}[1][]{\underline{\ormod} \mhyphen #1}

\DeclareMathOperator{\proj}{\operatorname{proj} \mhyphen \!}

\newcommand{\ul}[1]{\underline{#1}}
\newcommand{\gen}[1]{\langle #1 \rangle}

\newcommand{\lra}[1][]{\stackrel{#1}{\longrightarrow}}

\newcommand{\C}{\mathcal{C}}
\newcommand{\D}{\mathcal{D}}

\newcommand{\M}{\mathcal{M}}

\newcommand{\T}{\mathcal{T}}

\newcommand{\bL}{\mathbb{L}}

\newtheorem{therm}{Theorem}[section]
\newtheorem{defin}[therm]{Definition}
\newtheorem{propos}[therm]{Proposition}
\newtheorem{lemma}[therm]{Lemma}
\newtheorem{coro}[therm]{Corollary}
\newtheorem{rmk}[therm]{Remark}
\newtheorem{example}[therm]{Example}
\newtheorem{quest}{Question}

\begin{document}

\title{Some algebras that are not silting connected}
\author{Alex Dugas}
\address{Department of Mathematics, University of the Pacific, 3601 Pacific Ave, Stockton CA 95211, USA}
\email{adugas@pacific.edu}

%\subjclass[2010]{16G10, 16E35, 18E30}
%\keywords{derived equivalence, stable equivalence, quiver mutation, weakly symmetric algebra}

\begin{abstract} We give examples of finite-dimensional algebras $A$ for which the silting objects in $K^b(\proj A)$ are not connected by any sequence of (possibly reducible) silting mutations.  The argument is based on the fact that silting mutation preserves invariance under twisting by a fixed algebra automorphism, combined with the existence of spherical modules that are not invariant under such a twist.
\end{abstract}

\maketitle
\section{Introduction}

%%% Test typesetting %%%%

%The category $\rmod A$ is the same as the category $\rmod[A]$, while the category $A \lmod$ is the same as the category $\lmod[A]$.\\
%The category $\stmod A$ is the same as $\stmod[A]$ and the category $A \stlmod$ is the same as $\stlmod[A]$.\\

In \cite{AI}, Aihara and Iyama develop the theory of mutation for silting subcategories of a triangulated category.  One of the principal settings in which silting mutation is of interest is the category $K^b(\proj A)$ of perfect complexes over a finite-dimensional algebra $A$.  Here, silting complexes provide a nice generalization of tilting complexes that are better behaved under mutation.  In particular, it is always possible to mutate a silting complex at any one of its direct summands to obtain a new silting complex; whereas the same is not necessarily true when one considers only tilting complexes.  The action of irreducible silting mutation on the set of silting objects in $K^b(\proj A)$ can be visualized by the silting quiver of $A$, which also coincides with the Hasse diagram for a natural partial order on the set of silting objects.
% as the quiver with vertices corresponding to isomorphism classes of basic silting objects in $K^b(\proj A)$ and arrows corresponding to irreducible (left) silting mutations.

  Aihara and Iyama propose the problem of determining which algebras have a connected silting quiver.
  %, or equivalently, for which algebras the Hasse diagram of the poset of silting objects in $K^b(\proj A)$ is connected.
  Such algebras have come to be termed {\it silting connected} \cite{Ai}.   While several classes of algebras -- including representation finite symmetric algebras, local algebras and piecewise hereditary algebras -- are known to be silting connected, less is known about which algebras fail to be silting connected.  Two examples of symmetric algebras that are not silting connected, one originally discovered by Aihara, Grant and Iyama, appear in \cite[\S 6.3]{DIJ}.  However, in each example the silting objects are all linked by combinations of irreducible silting mutation and powers of the suspension, which is in fact a reducible silting mutation corresponding to the zero summand of the silting object.  Aihara and Iyama state that they are aware of no algebras $A$ where (not necessarily irreducible) iterated silting mutations do not act transitively on the set of basic silting objects in $K^b(\proj A)$.  This sparsity of known examples is likely due to the difficulty in showing that {\it no} possible sequence of mutations can change one silting complex into another; as opposed to such examples being uncommon.

  The purpose of this note is to present a family of examples where iterated silting mutation does not act transitively on the set of silting objects in $K^b(\proj A)$.  In fact, we see that the silting quivers in our examples will have infinitely many connected components, even if one includes edges for reducible silting mutations.  Our proof that these algebras are not silting connected makes use of a rather elementary observation (Proposition~\ref{prop:InvariantSilting}), which says that if each indecomposable object in a silting subcategory $\M$ is invariant under an automorphism of the ambient triangulated category $\T$, then the same is true for any mutation of $\M$.  We apply this fact to the algebras under consideration by showing that they admit spherical modules, and hence tilting complexes associated to the corresponding spherical twists, that are not invariant under such an automorphism.

\section{Silting Mutation}

In this section we review the definition of silting mutation in a triangulated category $\T$ due to Aihara and Iyama \cite{AI}, and show that the class of silting subcategories in which every object is invariant under a fixed automorphism of $\T$ is stable under mutation.

Throughout this section $\T$ will be a triangulated category, with suspension functor denoted $[1]$.  When we speak of a subcategory of $\T$ we shall always mean a strict, full subcategory closed under finite direct sums and direct summands.  Recall that a subcategory $\M$ of $\T$ is {\bf silting} if $\T(\M,\M[>0]) = 0$ and $\M$ generates $\T$ as a triangulated category.  We let $\silt(\T)$ denote the collection of all silting subcategories of $\T$.  If $\D$ is a covariantly (resp. contravariantly) finite subcategory of $\M$, then the left (resp. right) mutation of $\M$ at $\D$ is defined as the subcategory 
$$\mu^+(\M;\D) = \add (\D \cup \{N_M \mid M \in \M\}) \ \
 (\mbox{resp.}\ \mu^-(\M;\D) = \add(\D \cup \{L_M \mid M \in \M\}))$$
where $N_M$ (resp. $L_M$) is defined as the cone (resp. co-cone) of a left (resp. right) $\D$-approximation of $M$.  In other words, $N_M$ and $L_M$ are defined by distinguished triangles of the form
$$M \stackrel{f}{\to} D \to N_M \to M[1] \ \ \mbox{and} \ \ L_M \to D' \stackrel{f'}{\to} M \to L_M[1]$$
where $D, D' \in \D$ and $f$ and $f'$ are left and right $\D$-approximations, respectively.  An important point here is that these mutated subcategories do not depend on the choices of $\D$-approximations in their definition.

Aihara and Iyama show that for any silting subcategory $\M$ and any covariantly (resp. contravariantly) finite subcategory $\D$, the mutation $\mu^+(\M;\D)$ (resp. $\mu^-(\M,\D)$) is also silting and satisfies 
$$\mu^-(\mu^+(\M;\D);\D) = \M \ \ \ (\mbox{resp.}\ \mu^+(\mu^-(\M;\D);\D) = \M).$$

In addition, if $\T$ is a Krull-Schmidt triangulated category, then Aihara and Iyama define {\bf irreducible} silting mutations of a silting subcategory $\M$ as those mutations with respect to a subcategory $\D$ for which $\M \setminus \D$ contains a unique indecomposable object (up to isomorphism).  As usual, we write $\ind \C$ for the set of isomorphism classes of indecomposable objects in a category $\C$.  Under the assumption 
\begin{equation}\tag{F}\parbox{\dimexpr\linewidth-4em}{$\T$ is Krull-Schmidt and for any silting subcategory $\M$ of $\T$ and any $X \in \ind \M$, the subcategory $\M_X := \add( (\ind\ \M) \setminus X)$ is functorially finite in $\M$;}\end{equation}
irreducible left and right silting mutations are defined for all silting subcategories of $\T$.  This assumption holds, in particular, if $\T$ is a Hom-finite $k$-category (over a field $k$) and has a silting object.  In this case, we abbreviate $\mu^\pm(\M;\M_X)$ by $\mu^\pm_X(\M)$ for indecomposable objects $X$ in $\M$.

Finally, recall the definition of the silting quiver $Q(\silt(\T))$ of $\T$, under the assumption (F).  Its vertices are the (equivalence classes of) basic silting objects in $\T$, and each $M = \oplus_{i=1}^r M_i \in \silt \T$ has arrows to $\mu^+_{M_i}(M)$ for each $1 \leq i \leq r$.  Thus, two silting objects are in the same connected component of this quiver if and only if they are linked by iterated irreducible (left or right) silting mutation.  We point out that the silting quiver of $\T$ only conveys information about {\it irreducible} silting mutation in $\T$, and the silting connectedness of $\T$ usually refers to the connectedness of this quiver, i.e., to the connectedness of $\silt(\T)$ under irreducible mutation.  In this paper, however, we are really interested in the connectedness of $\silt(\T)$ with respect to all mutations.  Thus we will consider the {\it extended silting quiver} of $\T$, obtained from $Q(\silt(\T))$ by adding arrows $M \to \mu^+(M;\D)$ for each silting object $M$ and each subcategory $\D = \add(M')$ for a basic direct summand $M'$ of $M$.

%In the existing literature, silting mutation often refers only to irreducible silting mutation, and the silting quiver of $\T$ only reflects irreducible silting mutation.

Now consider an automorphism $\alpha$ of $\T$.  We say that $X \in \T$ is $\alpha$-{\bf invariant} if $\alpha X \cong X$.  We say that a subcategory $\C$
%want/need other assumptions on $\C$? eg. additive, closed under summands?
 of $\T$ is $\alpha$-{\bf invariant} if each object $X \in \C$ is $\alpha$-invariant.  %By contrast, we say that $\C$ is $\alpha$-{\bf stable} if $X^\alpha$ belongs to $\C$ for each $X \in \C$.  We also say that $\C$ is {\bf strongly $\alpha$-invariant} if $\C$ is $\alpha$-stable and $\alpha|_\C \cong I_\C$ where $I_\C$ is the identity functor on $\C$.  
 %Note that a strongly $\alpha$-invariant subcategory of $\T$ is a triangulated subcategory, although the same need not be true of a merely $\alpha$-invariant subcategory.  %Is this corrrect? Examples?  No, not correct unless it is the subcategory of all $\alpha$-invariant objects.

\begin{propos}\label{prop:InvariantSilting}  Let $\T$ be a Krull-Schmidt category 
% This assumption shouldn't be necessary, but don't see how to argue without it.  Removing it should be related to showing that the choice of the approximation doesn't matter in definition of mutation.
and let $\M$ be a silting subcategory of $\T$ that is $\alpha$-invariant.  Then for any covariantly (resp. contravariantly) finite subcategory $\D$ of $\M$ the silting subcategory $\mu^+(\M;\D)$ (resp. $\mu^-(\M;\D)$) is also $\alpha$-invariant.  
% Is the same true for strongly $\alpha$-invariant?
\end{propos}

\begin{proof} We give the proof for left silting mutation only, as the other half of the argument is dual.  Consider a triangle $$M \stackrel{f}{\to} D \stackrel{g}{\to} N_M \to M[1]$$ with $M \in \M$, $D \in \D$ and $f$ a minimal left $\D$-approximation.  Since $\D$ is $\alpha$-invariant, applying $\alpha$ yields another minimal left $\D$-approximation  
$$\alpha M \stackrel{\alpha f}{\to} \alpha D \stackrel{\alpha g}{\to} \alpha N_M \to \alpha M[1].$$
Since $M$ is $\alpha$-invariant there is an isomorphism $u: M \to \alpha M$, and uniqueness of minimal approximations implies the existence of an isomorphism $D \to \alpha D$ making the diagram commute
$$\xymatrix{M \ar[r]^f \ar[d]^\cong_u & D \ar[r]^g \ar[d]^{\cong} & N_M \ar@{-->}[d]^{\cong} \ar[r] & M[1] \ar[d] \\
\alpha M \ar[r]_{\alpha f} & \alpha D \ar[r]_{\alpha g} & \alpha N_M \ar[r] & \alpha M[1].}$$
Thus there is an induced isomorphism from $N_M$ to $\alpha N_M$ as required.  Since any $N \in \ind\ \mu^+(\M;\D) \setminus \ind \D$ arises in this way (by choosing a right $\D$-approximation of $N$ to get a left $\D$-approximation of $L_N \in \ind \M$ with cone $N$), we see that each indecomposable in $\mu^+(\M;\D)$ is $\alpha$-invariant.
\end{proof}

As a consequence, we obtain the following criterion for $\T$ to fail to be silting connected.  
%We will give examples where it can be applied in the next section.

\begin{coro}\label{coro:SiltingDisconnected} Assume that $\T$ has an $\alpha$-invariant silting subcategory $\M$ and another silting subcategory $\M'$ that is not $\alpha$-invariant.  Then $\M'$ cannot be obtained from $\M$ by iterated silting mutation.  In particular, the action of iterated silting mutation on $\silt(\T)$ is not transitive. 
\end{coro}

%\section{A criterion for silting disconnectedness}

\section{Spherical modules}

In this section we develop some simple properties of spherical modules in the derived category of a finite-dimensional algebra of finite global dimension.  First recall that if $\T$ is a Hom-finite triangulated $k$-category with Serre functor $S$, then an object $E$ in $T$ is $d$-{\it spherical} if $S(E) \cong E[d]$ and $$\T(E,E[j]) \cong \left\{ \begin{array}{rl} k, & \mbox{if\ }j = 0,d \\ 0, & \mbox{otherwise} \end{array} . \right.$$
We let $A$ be a finite-dimensional $k$-algebra of finite global dimension, and write $\rmod A$ for the category of finitely-generated right $A$-modules.  We will focus on the case where $\T$ is the bounded derived category $D^b(\rmod A)$, which we may also abbreviate as $D^b(A)$.  Recall that $D^b(\rmod A)$ has a Serre functor $S \cong -\otimes^\bL_A DA$ where $D$ denotes the duality $\Hom_k(-,k)$. 

Associated to any $d$-spherical object $E$ in $D^b(\rmod A)$, Seidel and Thomas have defined an exact auto-equivalence $\Phi_E$ of $D^b(\rmod A)$, known as a {\bf spherical twist} \cite{SeTh}.  For any $X$ in $D^b(\rmod A)$ we can compute $\Phi_E(X)$ as the cone of the natural evaluation map $ev_X$ in the distinguished triangle below
$$\coprod_{j \in \mathbb{Z}} \Hom_{D^b(A)}(E[j],X) \otimes_k E[j] \lra[ev_X] X \to \Phi_E(X) \to .$$
One easily checks that $\Phi_E(E) \cong E[1-d]$.  Additionally, because $\Phi_E$ is an auto-equivalence $\Phi_E(A)$ will be a tilting complex with endomorphism ring isomorphic to $A$.  In order to say more about $\Phi_E(A)$, we now assume that $E \in \rmod A$, i.e., that $E$ is a spherical module.  We also let $e_1, \ldots, e_n$ be a complete set of pairwise orthogonal primitive idempotents for $A$.

\begin{lemma}\label{lem:Emaps}  Assume the $A$-module $E$ is $d$-spherical in $D^b(\rmod A)$.  Then for all $m, j \in \mathbb{Z}$ and each $i$ with $1 \leq i \leq n$ we have isomorphisms
$$\Hom_{D^b(A)}(E[j],\Phi_E^m(e_iA)) \cong \left\{ \begin{array}{rl} D(Ee_i),& \mbox{if}\ j = m(1-d)-d \\ 0, & \mbox{otherwise}
\end{array} . \right.$$

\end{lemma}

\begin{proof}  We prove the claim by induction on $m$.  For $m=0$, we have by Serre duality
\begin{eqnarray*} \Hom_{D^b(A)}(E[j],e_iA) & \cong & D(\Hom_{D^b(A)}(e_iA,E[j+d])) \\
& \cong & D(H^{j+d}(E)e_i). \end{eqnarray*}
Now one uses that $H^*(E) \cong H^0(E) \cong E$.

Now assume the claim holds for some $m \geq 0$.  We have 
\begin{eqnarray*} \Hom_{D^b(A)}(E[j],\Phi_E^{m+1}(e_iA)) & \cong & 
\Hom_{D^b(A)}(\Phi_E(E[j+d-1]), \Phi_E(\Phi_E^m(e_iA))) \\
& \cong & \Hom_{D^b(A)}(E[j+d-1],\Phi_E^m(e_iA)),
\end{eqnarray*}
which vanishes unless $j+d-1 = m(1-d)-d$, or equivalently $j = (m+1)(1-d)-d$, in which case it is isomorphic to $D(Ee_i)$ as a $k$-vector space.  The proof for $m<0$ is similar.
\end{proof}

Next we describe the homology of the iterated spherical twists of each indecomposable projective $e_iA$.

\begin{lemma}\label{lem:Homology} Assume that $E_A$ is $d$-spherical in $D^b(\rmod A)$ for some $d \geq 2$.  Then for all $m, j \in \mathbb{Z}$ with $m \geq 0$ and each $i$ with $1 \leq i \leq n$ we have
$$H^j(\Phi_E^m(e_iA)) \cong \left\{ \begin{array}{rl} e_iA, & \mbox{if}\ j=0 \\ D(Ee_i) \otimes_k E, & \mbox{if}\ j = l(d-1)\ \mbox{for}\ 1\leq l \leq m \\ 0, & \mbox{otherwise} \end{array} . \right.$$
\end{lemma}

\begin{proof}  We again argue by induction on $m$.  For $m=0$, the claim is trivial.  Now assume that it holds for some $m \geq 0$, and consider the triangle used to define $\Phi_E(\Phi_E^m(e_iA))$

$$ D(Ee_i) \otimes_k  E[m(1-d)-d] \to \Phi_E^m(e_iA) \to \Phi_E^{m+1}(e_iA) \to D(Ee_i) \otimes_k E[(m+1)(1-d)].$$

The corresponding long exact sequence in homology (obtained by applying $\Hom_{D^b(A)}(A,-)$) shows that $$H^j(\Phi_E^{m+1}(e_iA)) \cong H^j(\Phi_E^m(e_iA))$$ for all $j \neq (m+1)(d-1), (m+1)(d-1)+1$.  For the remaining $j$, we know that $H^j(\Phi_E^m(e_iA))$ vanishes by the induction hypothesis (this is where we need $d \geq 2$).  Thus $$H^j(\Phi_E^{m+1}(e_iA)) \cong H^j(D(Ee_i) \otimes_k E[(m+1)(1-d)]),$$ which is either $E[(m+1)(1-d)] \otimes_k E$ or $0$ depending on whether $j =(m+1)(d-1)$ or not.
\end{proof}

\section{Examples}

The goal of this section is to describe concrete examples of finite-dimensional algebras $A$ over a field $k$ for which iterated silting mutation does not act transitively on the set of (equivalence classes of) silting objects in $K^b(\proj A)$.  To this end, we fix an integer $n \geq 2$ and let $A = A_n$ be the path algebra of the following quiver $Q = Q_n$
$$\xymatrix{1 \ar[r]<0.5ex>^x \ar[r]<-0.5ex>_y & 2 \ar[r]<0.5ex>^x \ar[r]<-0.5ex>_y & \cdots \ar[r]<0.5ex>^x \ar[r]<-0.5ex>_y & n}$$
modulo the relations $x^2=y^2=0$.  We write $e_i$ for the primitive idempotent of $A$ corresponding to vertex $i$ (for $1 \leq i \leq n$), and $S_i, P_i$ and $I_i$ for the corresponding simple, indecomposable projective and indecomposable injective right $A$-modules, respectively.
It is not hard to see that $\gldim A = n-1$. % In particular we may identify $K^b(\proj A)$ with the derived category $D^b(\rmod A)$.  Recall that $D^b(\rmod A)$ has a Serre functor $S \cong -\otimes^L_A DA$ where $D$ denotes the duality $Hom_k(-,k)$. 

We let $\epsilon \in \aut_k(A)$ be the order two automorphism induced by the automorphism of $Q$ that fixes each vertex and swaps each pair of $x$ and $y$ arrows.  We view $\epsilon$ as acting on $A$ on the right, so that it induces an automorphism $\alpha$ of $\rmod A$ (acting on the left).  By definition, $\alpha$ sends a right $A$-module $M$ to the twisted module $\alpha M = M^\epsilon$ which equals $M$ as an abelian group and has $A$-action given by $m \cdot a = ma^{\epsilon}$ for all $m \in M^\epsilon=M$ and all $a \in A$.  Since $\alpha M = M$ as sets for any module $M$, we can define $\alpha(f) = f$ (as functions) for any morphism $f$.  This action restricts to an automorphism of $\proj A$ and hence also induces automorphisms of $K^b(\proj A)$ and $D^b(\rmod A)$, which we continue to write as $\alpha$.  Since $e_i^\epsilon = e_i$ for all $i$, it is clear that each indecomposable projective $P_i$ is $\alpha$-invariant.

We set $E = e_1A/e_1yA$, which is a uniserial module of length $n$.  Note that $\alpha E \cong e_1A/e_1xA \ncong E$, so $E$ is not $\alpha$-invariant.  When we consider $E \in D^b(\rmod A) \approx K^b(\proj A)$ we may identify $E$ with its minimal projective resolution 
$$P_E = 0 \to P_n \stackrel{y\cdot}{\to} P_{n-1} \stackrel{y\cdot}{\to} \cdots \to P_2 \stackrel{y\cdot}{\to} P_1 \to 0,$$
which has $P_1$ in degree $0$.

\begin{propos}\label{prop:ESpherical}
\begin{enumerate}
\item If $n$ is even, then $E$ and $\alpha E$ are Hom-orthogonal $(n-1)$-spherical objects in $D^b(\rmod A)$.  
%\item If $n$ is even, we have $\Hom_{D^b(A)}(E, \alpha E[j]) = 0$ for all $j$.
 \item If $n$ is odd, then $E$ is exceptional in $D^b(\rmod A)$ and $S(E) \cong \alpha E[n-1]$.  
 %Recall that the former means that $\Hom_{D^b(\rmod A)}(E,E[i]) \cong \left\{ \begin{array}{rl} k, & \mbox{if\ }i = 0 \\0, & \mbox{otherwise} \end{array} \right.$
\end{enumerate}
\end{propos}

\begin{proof}
To compute $S(E)$ we tensor the projective resolution of $E$ with $DA$ to obtain the complex
$$0 \to I_{n} \stackrel{d_{n-1}}{\to} I_{n-1} \stackrel{d_{n-2}}{\to} \cdots \stackrel{d_{1}}{\to} I_1 \to 0$$
with $d_i = D(\cdot y)$ where $\cdot y$ denotes the map $Ae_i \to Ae_{i+1}$ given by right multiplication by $y$.  A simple calculation shows that $d_i$ maps the left branch of $I_{i+1}$ to $0$ and maps the right branch of $I_{i+1}$ onto the left branch of $I_i$ as in the figure:

%$$\xymatrixrowsep{0.5pc} \xymatrixcolsep{0.3pc} \xymatrix{ 1 \ar@{-}[d] & & 1 \ar@{-}[d] \ar@{|->}[rr]^{d_i} & & 1 \ar@{-}[d] & & 1 \ar@{-}[d] \\			 2 \ar@{-}[d] & & 2 \ar@{-}[d] \ar@{|->}[rr] & & 2 \ar@{-}[d]   & & 2 \ar@{-}[d] \\			 \vdots \ar@{-}[d] & & \vdots \ar@{-}[d] & &  \vdots \ar@{-}[d] & & \vdots \ar@{-}[d] \\			 i-1 \ar@{-}[d]_y & & i-1 \ar@{-}[d]^x \ar@{|->}[rr] & & i-1 \ar@{-}[dr]_x & & i-1 \ar@{-}[dl]^y \\			 i \ar@{-}[dr]_x  & & i \ar@{-}[dl]^y \ar@{|->}[rr] & & & i\\			 & i+1 \\ & I_{i+1} \ar[rrrr]_{d_i} & & & & I_i}$$
			 
 $$\xymatrixrowsep{0.6pc} \xymatrixcolsep{0.5pc} \xymatrix{ 1 \ar@{-}[d]  & \ 1\ \ar@{-}[d]  \ar@{|-->}[rr]^{d_i} & & \ 1 \ \ar@{-}[d] &  1 \ar@{-}[d] \\
			 2 \ar@{-}[d]  & \ 2\ \ar@{-}[d]   \ar@{|-->}[rr] & & \ 2 \  \ar@{-}[d]    & 2 \ar@{-}[d] \\
			 \vdots \ar@{-}[d]  & \vdots \ar@{-}[d]   & & \vdots \ar@{-}[d]  & \vdots \ar@{-}[d] \\
			 i-1 \ar@{-}[d]_y & \ i-1 \ \ar@{-}[d]^x    \ar@{|-->}[rr] & & \ i-1 \ \ar@{-}[d]_x  & i-1 \ar@{-}[dl]^y \\
			 i \ar@{-}[d]_x   & \ i\  \ar@{-}[dl]^y   \ar@{|-->}[rr]  & &\  i \ \\
			  i+1 \\  I_{i+1} \ar[rrr]_{d_i} & & &   I_i}$$

It is clear that this complex is exact, except in degree $1-n$, where its homology is $\ker d_{n-1}$.  This is a length $n$ uniserial module with a length two submodule annihilated by $y$.  If $n$ is even, this uniserial module is isomorphic to $E$, while if $n$ is odd it is isomorphic to $\alpha E$.  Hence
$$S(E) \cong \left\{ \begin{array}{rl} E[n-1], & \mbox{if}\ n\ \mbox{is\ even} \\ \alpha E[n-1], & \mbox{if}\ n\ \mbox{ is\ odd} \end{array} \right.$$

To compute $\Hom^*_{D^b(A)}(E, E) \cong \Ext^*_A(E,E)$ we apply $\Hom_A(-,E)$ to the projective resolution of $E$.  Note that $\Hom_A(e_iA,E) \cong Ee_i$ is one-dimensional for each $i$.  Moreover, the induced map 
$$\Hom_A( y \cdot, E) : \Hom_A(e_iA,E) \to \Hom_A(e_{i+1},E)$$ can be identified with the map given by right multiplication by $y$ from $Ee_i$ to $Ee_{i+1}$.  Thus this map is an isomorphism if $i$ is even or else the zero map when $i$ is odd.  It follows that the complex $\Hom_A(P_E,E)$ has nonzero homology (isomorphic to $k$) only in degree $0$ if $n$ is odd, or else only in degrees $0$ and $n-1$ if $n$ is even.  Similarly if we apply $\Hom_A(-,\alpha E)$ to $P_E$, we get a complex of one-dimensional vector spaces with maps corresponding to right multiplication by $y$ from $\alpha E e_i$ to $\alpha E e_{i+1}$.  This time these maps are isomorphisms whenever $i$ is odd and zero otherwise.  Thus we see that $\Hom_{D^b(A)}(E, \alpha E[j]) \cong \Ext^j_A(E, \alpha E) = 0$ for all $j$ in case $n$ is even.  For completeness, we also note that if $n$ is odd, we have $\Hom_{D^b(A)}(E, \alpha E[j]) \cong \Ext^j_A(E,\alpha E) = 0$ for all $j\neq n-1$, while $\Ext^{n-1}_A(E,\alpha E) \cong k$.
\end{proof}

\begin{rmk}\label{rmk:DihedralConnection} \emph{When $n=2q$ is even, $A$ is the Beilinson algebra (see \cite{Ch}) of the dihedral algebra $$\Lambda = k\gen{x,y}/(x^2,y^2,(xy)^q-(yx)^q)$$ with the usual grading that places $x$ and $y$ in degree $1$.  This connection provides another way to see that $E$ is a spherical object and to study the action of the spherical twist $\Phi_E$ using results of  \cite{Dug, DuTr}.  We will elaborate in the Appendix.}
%\ldots.  Moreover, the spherical twists $\Phi_E$ and $\Phi_{E^\alpha}$ correspond to the spherical stable twists $\sigma_x$ and $\sigma_y$ on the stable category of (graded) $\Lambda$-modules studied in \cite{Dug, DuTr}.}
%Elaborate
\end{rmk}

\begin{rmk}\label{rmk:ExceptionalCycle} \emph{When $n$ is odd, part (b) of the above Proposition shows that $(E, \alpha E)$ is an {\it exceptional 2-cycle}, following the terminology of \cite{BPP}, with $k_1 = k_2 = n-1$. It is shown in \cite{BPP} that the twist with respect to the direct sum of objects in an exceptional cycle is an auto-equivalence of $D^b(\rmod A)$.  However, since $E \oplus \alpha E$ is $\alpha$-invariant, this twist will only produce other $\alpha$-invariant tilting complexes.}
\end{rmk} 
 
 Now assume that $n \geq 4$ is even.  As we saw in the previous section, there is an auto-equivalence $\Phi_E$ of $D^b(\rmod A)$.  Naturally, it will induce an automorphism of the silting quiver of $A$, which takes $A$ to the tilting complex $\Phi_E(A)$.  Likewise applying powers of $\Phi_E$ (or its quasi-inverse) to $A$ yields tilting complexes $\Phi_E^m(A)$ for each $m \in \mathbb{Z}$.
 
\begin{propos}\label{prop:MainResult}
Let $A = A_n$ and $E$ be as above, with $n \geq 4$ even.  Then for any distinct integers $j, m$ the tilting complexes $\Phi_E^j(A)$ and $\Phi_E^m(A)$ are not connected by iterated silting mutation.  In particular, the extended silting quiver of $K^b(\proj A)$ has infinitely many connected components.
\end{propos} 
 
\begin{proof}  Since $Ee_i$ is one-dimensional for each $i$, by Lemma~\ref{lem:Homology} we have $H^{n-2}(\Phi_E(A)) \cong E^n$, which is not $\alpha$-invariant.  Thus, by Corollary~\ref{coro:SiltingDisconnected}, $\Phi_E(A)$ is not connected to $A$ by iterated silting mutation.  Similarly, no $\Phi_E^m(A)$ can be connected to $A$ for $m \neq 0$.  Consequently, for any $m \neq j$,  the tilting complexes $\Phi_E^j(A)$ and $\Phi_E^m(A)$ are not connected by iterated silting mutation.  For if they were, applying $\Phi_E^{-j}$ would yield a way of connecting $A$ and $\Phi_E^{m-j}(A)$ by iterated silting mutation, which we know is not possible.
\end{proof}
 
By Proposition~\ref{prop:ESpherical}, $E$ and $\alpha E$ are Hom-orthogonal spherical objects.  Hence, by \cite{SeTh}, the equivalences $\Phi_E$ and $\Phi_{\alpha E}$ commute.  Moreover, it is easy to see that $\Phi_{\alpha E} \cong \alpha \Phi_E \alpha^{-1}$.  We further note that $\Phi_{\alpha E}\Phi_E(A)$ is again $\alpha$-invariant.  However, we do not know if it is connected to $A$ by iterated (irreducible) silting mutation.  One can check further that 
$\Phi_{\alpha E}\Phi_E(A) \cong \tau^{-1}(A)$ in $K^b(\proj A)$, where $\tau$ denotes the Auslander-Reiten translation.  We verify this for $n=4$ by a direct calculation below, and justify it more generally in the Appendix.

\begin{example} \emph{To give a more concrete illustration of some of these tilting complexes which are not connected by silting mutation, we now describe the tilting complexes $\Phi_E(A)$ and $\Phi_{\alpha E} \Phi_E(A)$ when $n=4$.  For each $1 \leq i \leq 4$, $\Phi_E(e_iA)$ can be described as a mapping cone of a map $E[-3] \to e_iA$. In the complexes below, we indicate the degree-$0$ term by underlining it.}
 
$$\xymatrix{\Phi_E(e_4A) \ \cong & 0 \ar[r] & 0 \ar[r] & \ul{e_3A} \ar[r]^{y}  & e_2A \ar[r]^{y} &  e_1A \ar[r] & 0 \\
 \Phi_E(e_3A) \ \cong & 0 \ar[r] & e_4A \ar[r]^(.4){\scriptsize \begin{pmatrix}x \\ y \end{pmatrix}} & \ul{(e_3A)^2} \ar[r]^(.55){\scriptsize \begin{pmatrix}0 & y\end{pmatrix}} & e_2A \ar[r]^{y} & e_1 A \ar[r] & 0 \\
 \Phi_E(e_2A) \ \cong & 0 \ar[r] & e_4A \ar[r]^(0.35){\scriptsize \begin{pmatrix}yx \\ y \end{pmatrix}} & \ul{e_2A \oplus e_3 A} \ar[r]^(0.65){\scriptsize \begin{pmatrix}0 & y\end{pmatrix}} & e_2A \ar[r]^{y} & e_1 A \ar[r] & 0 \\
\Phi_E(e_1A) \ \cong & 0 \ar[r] &  e_4A \ar[r]^(0.35){\scriptsize \begin{pmatrix}xyx \\ y \end{pmatrix}} & \ul{e_1A \oplus e_3 A} \ar[r]^(0.65){\scriptsize \begin{pmatrix}0 & y\end{pmatrix}} & e_2A \ar[r]^{y} & e_1 A \ar[r] & 0}$$
 
\emph{Each complex has homology isomorphic to $e_iA$ in degree $0$ and to $E$ in degree $2$.  In particular they are not $\alpha$-invariant.}

 \emph{Next we compute $\Phi_{\alpha E}$ of each of the above complexes.  For each $i$, $\Phi_{\alpha E}(\Phi_E(e_iA))$ is realized as the mapping cone of the unique (up to a scalar multiple) map $\alpha E[-3] \to \Phi_E(e_iA)$.}
 
$$\xymatrixcolsep{1.5pc} \xymatrix{
 \Phi_{\alpha E}\Phi_E(e_4A) \ \cong & 0 \ar[r] & e_4A \ar[rr]^{\scriptsize \begin{pmatrix} x \\ y \end{pmatrix}} & &
 	\ul{(e_3A)^2} \ar[rr]^{\scriptsize \begin{pmatrix} x & 0 \\ 0 & y \end{pmatrix}} & &
 	(e_2A)^2 \ar[rr]^{\scriptsize \begin{pmatrix} x & 0 \\ 0 & y \end{pmatrix}} & &
 	(e_1A)^2 \ar[r] & 0 \\
 \Phi_{\alpha E}\Phi_E(e_3A) \ \cong & 0 \ar[r] & (e_4A)^2 \ar[rr]^{\scriptsize \begin{pmatrix}x & 0 \\ y & x \\ 0 & y \end{pmatrix}} & &
 	\ul{(e_3A)^3} \ar[rr]^{\scriptsize \begin{pmatrix} x & 0 & 0\\ 0 & 0 & y \end{pmatrix}} & &
 	(e_2A)^2 \ar[rr]^{\scriptsize \begin{pmatrix} x & 0 \\ 0 & y \end{pmatrix}} & &
 	(e_1 A)^2 \ar[r] & 0 \\
 \Phi_{\alpha E}\Phi_E(e_2A) \ \cong & 0 \ar[r] & (e_4A)^2 \ar[rr]^(0.35){\scriptsize \begin{pmatrix}x & 0 \\ xy & yx \\ 0 & y \end{pmatrix}} & &
 	\ul{e_3A \oplus e_2A \oplus e_3A} \ar[rr]^(0.6){\scriptsize \begin{pmatrix} x & 0 & 0\\ 0 & 0 & y \end{pmatrix}} & &
 	(e_2A)^2 \ar[rr]^{\scriptsize \begin{pmatrix} x & 0 \\ 0 & y \end{pmatrix}} & &
 	(e_1 A)^2 \ar[r] & 0 \\
 \Phi_{\alpha E}\Phi_E(e_1A) \ \cong & 0 \ar[r] & (e_4A)^2 \ar[rr]^(0.4){\scriptsize \begin{pmatrix}x & 0 \\ yxy & xyx \\ 0 & y \end{pmatrix}} & &
 	\ul{e_3A \oplus e_1 A \oplus e_3A} \ar[rr]^(0.6){\scriptsize \begin{pmatrix} x & 0 & 0\\ 0 & 0 & y \end{pmatrix}} & &
 	(e_2A)^2 \ar[rr]^{\scriptsize \begin{pmatrix} x & 0 \\ 0 & y \end{pmatrix}} & &
 	(e_1 A)^2 \ar[r] & 0 }$$
 
\emph{ Each of the above indecomposable complexes has homology isomorphic to $e_iA$ in degree $0$ and to $E \oplus \alpha E$ in degree $2$.  Furthermore, one can check that each is $\alpha$-invariant.  We do not know if the corresponding tilting complex is connected to $A$ via silting mutation.  However, as mentioned earlier, we can observe that $S(\Phi_{\alpha E}\Phi_E(e_iA)) \cong \Phi_{\alpha E}\Phi_E(e_iA) \otimes_A DA$ gives the injective coresolution of $e_iA[1]$.  Hence} $$\Phi_{\alpha E}\Phi_E(A) \cong S^{-1}(A)[1] \cong \tau^{-1}(A).$$ 
 \end{example}

 \begin{quest} For the algebra $A$ above, is $A$ connected to $\Phi_{E^\alpha}\Phi_E(A) \cong S^{-1}(A)[1]$ by (irreducible) silting mutation?
 \end{quest}

% \begin{quest} For an arbitrary algebra $A$ (of finite global dimension), when is $A$ connected to $A[1]$, or to $DA$, by irreducible silting mutation?\end{quest}
  
%The trivial extension of the kronecker algebra is an example where $A$ is not connected to $A[1]$ by  irreducible silting mutation.  The same could be true for $T(A)$ with $A$ in this paper.  

% \begin{quest} For an arbitrary algebra $A$, is every silting object $M \in K^b(\proj A)$ $M$ connected to a tilting object by iterated (irreducible) silting mutation?\end{quest}

%This is (part of) one of Aihara's sufficient conditions for silting connectedness

%\begin{quest} Let $A$ be an algebra and $T(A)$ its trivial extension.  Rickard shows that if $P$ is a tilting complex in $K^b(\proj A)$, then $P \otimes_A T(A)$ is a tiliting complex in $K^b(\proj T(A))$.  It is easy to see that this is no longer true for silting complexes: if $P$ is silting, then $P \otimes_A T(A)$ is not silting in general.  However, silting complexes for $A$ and tilting complexes for $T(A)$ each have the same number of indecomposable summands and allow for mutation at any summand.  Is there some natural correspondence between these two sets that respects mutation?
%\end{quest}

%Kronecker algebra $A$ and its trivial extension show that tilting mutation over $T(A)$ is more complicated than silting mutation over $A$ \cite{DIJ}.  But perhaps there is still a map from one quiver to the other?
 
The same idea we have used here to show that the algebra $A$ is not silting connected can also be applied to show that its trivial extension $T(A)$ is not tilting connected (note that every silting complex in $K^b(\proj T(A))$ is tilting since $T(A)$ is a symmetric algebra).  The quiver of $T(A)$ is obtained from the quiver of $A$ by adding two arrows, also labeled $x$ and $y$, from vertex $n$ back to vertex $1$.  In addition to the relations $x^2 = 0$ and $y^2 = 0$ of $A$, which now extend to include the new arrows as well, $T(A)$ also has the relations $(xy)^q = (yx)^q$ at each vertex.  In particular, there is an order-two automorphism of $T(A)$ that swaps $x$ and $y$ at each vertex.  Abusing notation, will continue to write $\epsilon$ for this automorphism and $\alpha$ for the induced automorphisms of $\rmod T(A)$ and $K^b(\proj T(A))$.

Rickard \cite{DCSE} has shown that if $P$ is a tilting complex in $K^b(\proj \Lambda)$ for any finite-dimensional $k$-algebra $\Lambda$, then $P \otimes_\Lambda T(\Lambda)$ is a tilting complex in $K^b(\proj T(\Lambda))$.  Moreover, the endomorphism ring of $P \otimes_\Lambda T(\Lambda)$ can be identified with the trivial extension of $\Endo_{K^b(\Lambda)}(P)$.  Thus, for each $m \in \mathbb{Z}$, the induced tilting complex $F_m := \Phi_E^m(A) \otimes_A T(A)$ is not $\alpha$-invariant, and hence not connected to $T(A)$ through any sequence of tilting mutations, irreducible or otherwise.  Unfortunately, here it is not obvious whether each of these tilting complexes $F_m$ lies in a different connected component of the tilting quiver of $T(A)$.  For, while Rickard's results show that each $F_m$ induces an auto-equivalence $\rho_m$ of $D^b(\rmod T(A))$, we do not know whether $\rho_m \cong \rho_1^m$ for each $m$, and consequently we do not know if $\rho_m(F_j) \cong F_{j+m}$. One may check that it is no longer the case that these auto-equivalences $\rho_m$ are spherical twists.  One might wonder if the $\rho_m$, being derived auto-equivalences of symmetric algebras, could be instances of periodic twists studied by Grant \cite{Gra}.  However, this also appears to not be the case, since Grant shows that a periodic twist factors as a sequence of tilts by 2-term Okuyama-Rickard complexes \cite[Theorem B]{Gra}.  Consequently, any tilting complex associated to a periodic twist must be connected to $T(A)$ via tilting mutation.  We are thus not aware of whether the auto-equivalences $\rho_m$ of $D^b(\rmod T(A))$ have any characterization intrinsic to this derived category that does not rely on tilting complexes over $A$.
 
Finally, we recall that for any finite-dimensional $k$-algebra $\Lambda$, Koenig and Yang \cite{KY} have established mutation preserving bijections between (equivalence classes of) silting complexes in $K^b(\proj \Lambda)$, simple-minded collections in $D^b(\rmod \Lambda)$, bounded t-structures of $D^b(\rmod \Lambda)$ with length heart, and bounded co-t-structures of $K^b(\proj \Lambda)$.  Thus, the algebras $A$ as in Proposition~\ref{prop:MainResult} and their trivial extensions $T(A)$ yield examples of algebras $\Lambda$ for which mutation does not act transitively on simple-minded collections or on bounded t-structures with length heart in $D^b(\rmod \Lambda)$, or on bounded co-t-structures in $K^b(\proj \Lambda)$.

\section{Appendix: Connection to dihedral algebras}

The examples of spherical objects presented in this article were not discovered randomly, but rather correspond naturally to certain spherical stable twists introduced in \cite{Dug}.  While this correspondence is not necessary in the above exposition, it does present an alternative means of computing the actions of the spherical twist functors and also illustrates how further examples may be found.  For these reasons, we believe it is worthwhile to provide more details about this connection. 

We start by reviewing some general results of Happel and Chen, which we shall need.  For a $\mathbb{Z}$-graded algebra $\Gamma = \oplus_n \Gamma_n$, we write $\gr \Gamma$ for the category of finitely generated $\mathbb{Z}$-graded right $\Gamma$-modules and degree preserving morphisms, and $\stgr \Gamma$ for the associated stable category obtained by factoring out the ideal of morphisms that factor through a projective module.  For graded modules $X$ and $Y$, we will write $\hom_\Gamma(X,Y)$ (resp. $\sthom_{\Gamma}(X,Y)$) and $\ext^i_{\Gamma}(X,Y)$ for the spaces of degree-$0$ morphisms (resp. stable morphisms) and degree-$0$ extensions. If $\Gamma$ is concentrated in degrees $0$ through $c$ (with $\Gamma_c \neq 0$), Chen defines the {\it Beilinson algebra} of $\Gamma$ to be the matrix algebra $$B = b(\Gamma) = \begin{pmatrix} 
\Gamma_0 & \Gamma_1 & \Gamma_2 & \cdots & \Gamma_{c-2} & \Gamma_{c-1} \\
0 & \Gamma_0 & \Gamma_1 & \cdots & \Gamma_{c-3} & \Gamma_{c-2} \\ 
0 & 0 & \Gamma_0 & \cdots & \Gamma_{c-4} & \Gamma_{c-3} \\ 
 \vdots & \vdots & & \ddots & \vdots & \vdots \\ 
0 & 0 & 0 & \cdots & \Gamma_0 & \Gamma_1 \\
0 & 0 & 0 &\cdots & 0 & \Gamma_0 \end{pmatrix}.$$
Furthermore, $\Gamma$ is {\it well-graded} if $e\Gamma_c$ and $\Gamma_c e$ are nonzero for each primitive idempotent $e \in \Gamma_0$.  We note that $b(\Gamma)$ has finite global dimension if and only if $\Gamma_0$ does.  Finally, we note that one way of producing well-graded self-injective (in fact, symmetric) algebras is as trivial extensions.  If $B$ is any algebra, we will regard its trivial extension $T(B) = B \ltimes DB$ as a graded algebra with $B$ in degree $0$ and $DB$ in degree $1$.

\begin{therm}{\cite[Theorem 1.1]{Ch}} For a well-graded self-injective algebra $\Gamma$, there is an equivalence of categories $\gr T(b(\Gamma)) \approx  \gr \Gamma$.  (However, this equivalence typically does not commute with the grading shift.)
\end{therm}

Now, combining  with this Happel's theorem which states that for any algebra $B$ of finite global dimension there is an equivalence of triangulated categories $D^b(\rmod B) \approx \stgr T(B)$, we obtain the following.

\begin{coro}{\cite[Corollary 1.2]{Ch}} Let $\Gamma$ be a well-graded self-injective algebra such that $\Gamma_0$ has finite global dimension.  Then we have equivalences of triangulated categories
$$D^b(\rmod b(\Gamma)) \approx \stgr T(b(\Gamma)) \approx \stgr \Gamma.$$
\end{coro}

We continue to assume that $\Gamma$ is a well-graded self-injective algebra such that $\Gamma_0$ has finite global dimension, and we set $B = b(\Gamma)$.  
Each of the above categories has a Serre functor, and by uniqueness of Serre functors, these equivalences commute with these Serre functors up to natural isomorphism.  For $D^b(\rmod B)$, the Serre functor is given by $-\otimes^\bL_B DB$.  On $\stgr T(B)$, the Serre functor is given by $\nu \Omega$ where $\nu = -\otimes_{T(B)} D(T(B))$ is the Nakayama functor.  Observe that $D(T(B)) = D(B \oplus DB) \cong DB \oplus B$ with $DB$ in degree $0$ and $B$ in degree $-1$.  Hence $D(T(B)) \cong T(B)(1)$ as graded $(T(B),T(B))$-bimodules.  Consequently, $\nu_{T(B)}$ is isomorphic to the grading shift functor $-(1)$ on $\gr T(B)$.  Hence the Serre functor on $\stgr T(B)$ is isomorphic to $\Omega (1)$.  Passing to $\stgr \Gamma$, the Nakayama functor $\nu_{T(B)}$ will correspond to the grading shift $-(c)$, which we note is not the Nakayama functor of $\Gamma$, while the Serre functor on $\stgr \Gamma$ must be isomorphic to $\Omega (c)$. 

In \cite{Guo}, Guo uses covering theory to describe the relationship between $\Gamma$ and $T(b(\Gamma))$ more concretely.  Nameley, he defines an algebra $\Gamma^T$ as the orbit algebra of the category $\gr \Gamma$ with respect to the Nakayama automorphism.  If $\Gamma$ is symmetric, its Nakayama automorphism coincides with the grading shift $-(c)$, and hence in this case the algebra $\Gamma^T$ is a Galois covering of $\Gamma$ with group $\mathbb{Z}/c$.  Thus $\Gamma^T$ may also be realized as a smash product $\Gamma \# k(\mathbb{Z}/c)^*$ with respect to the natural grading on $\Gamma$ viewed as a $\mathbb{Z}/c$-grading.  By \cite[Theorems 5.8, 5.1]{Guo}, $\Gamma^T$ is isomorphic to a twisted trivial extension of $b(\Gamma)$.  If $\Gamma^T$ is also symmetric, it must be isomorphic to the (untwisted) trivial extension $T(b(\Gamma))$.

Thus, in case $\Gamma$ and $\Gamma^T$ are both symmetric, writing $\Psi$ for the equivalence $\gr T(b(\Gamma)) \to \gr \Gamma$, we have $$\Psi (M(1)) \cong \Psi(M)(c)$$ for all $M \in \gr T(b(\Gamma))$.

  We now return to the notation used earlier in the paper.  In particular $n = 2q$ is an even integer, $A =A_n$ is the algebra introduced in Section 4, and $$\Lambda = \Lambda_{2q} = k\gen{x,y}/(x^2,y^2,(xy)^q-(yx)^q)$$ is a local dihedral algebra.  Note that $\Lambda$ can be graded by placing $x$ and $y$ in degree $1$.  With this grading, $\Lambda$ is a well-graded symmetric algebra concentrated in degrees $0$ through $2q$.  By Theorem 5.1 of \cite{Guo}, it is easy to see that $A$ is the Beilinson algebra of $\Lambda$.  Moreover, here $\Lambda$ and $\Lambda^T$ are both symmetric, so we see that $T(A)$ is a Galois covering of $\Lambda$ with Galois group $\mathbb{Z}/n$, and the categories of $\mathbb{Z}$-graded modules over $T(A)$ and $\Lambda$ are equivalent, with the grading shift over $T(A)$ corresponding to the $n^{th}$ power of the grading shift over $\Lambda$.  

An object $X \in \stgr T(A)$ is $d$-Calabi-Yau if and only if $S(X) \cong X[d]$, or equivalently, $X(1) \cong \Omega^{-d-1}(X)$.  Equivalently, if we write $\tilde{X}$ for the corresponding object in $\stgr \Lambda$, we see that $\tilde{X}$ is $d$-Calabi-Yau if and only if $\tilde{X}(n) \cong \Omega^{-d-1}(\tilde{X})$.  The $(n-1)$-spherical object $E$ in $D^b(\rmod A)$ corresponds to the module $T = \Lambda/y\Lambda$ in $\stgr \Lambda$.  It is easy to see directly that $\Omega^{-1}(T) \cong T(1)$ as graded $\Lambda$-modules, and hence that $T$ is $(n-1)$-Calabi-Yau.  Furthermore, it is straightforward to verify that $\stEndo_{\Lambda}(T) \cong k[u]/(u^2)$ where $u$ is the unique (up to a scalar multiple) degree $2q-1$ endomorphism of $T$. Consequently, in $\stgr \Lambda$ we have
$$\ext^i_{\Lambda}(T,T) = \sthom_{\Lambda}(T,T[i]) = \sthom_{\Lambda}(T,T(i)) \cong \left\{ \begin{array}{ll} k, & i=0, n-1 \\
0, & \mbox{otherwise}.\end{array} \right.$$
We denote the corresponding spherical twist in $\stgr \Lambda$ by $\Phi_T$.  In \cite{Dug}, we defined a {\it spherical stable twist} $\sigma_y$ which is an auto-equivalence of the (ungraded) stable category $\stmod \Lambda$.  As both $\Phi_T$ and $\sigma_y$ are defined using cones of right $\add(T)$-approximations, it is clear that their actions agree on graded $\Lambda$-modules.  Likewise, the spherical object $\alpha E$ corresponds to $T' = \Lambda/x\Lambda$, and the spherical twist $\Phi_{T'}$ on $\stgr \Lambda$ can be viewed as a graded version of the spherical stable twist $\sigma_x$ on $\stmod \Lambda$.  Now, as $\sigma_x \sigma_y$ coincides with $\tau^{-1}$ on objects \cite[Example 7.1]{Dug}, it follows that $\Phi_T' \Phi_T$ also coincides with $\tau^{-1}$ on objects in $\stgr \Lambda$.  Carrying this information back to $D^b(A)$, we see that $\Phi_{\alpha E}\Phi_E$ also coincides with $\tau^{-1}$ on objects.  While we do not know if this gives a functorial factorization of $\tau^{-1}$, the fact that similar factorizations have appeared elsewhere (see \cite[Cor. 5.5]{BPP} for instance) suggests that something deeper may underlie this phenomenon.

%If $F : \stgr \Lambda \to \stmod \Lambda$ denotes the functor that forgets the grading, then one has $F \sigma_y \cong \Phi_T F$. In \cite{Dug} we describe the action of $\sigma_y$ on the component of the AR-quiver of $\Lambda$ containing the simple $\Lambda$-module.  Since the modules in this component are all gradable, there is an isomorphic component of the stable AR-quiver of $\gr \Lambda$.  


\begin{thebibliography}{99}
\bibitem{Ai} T. Aihara. \emph{Tilting connected symmetric algebras.}  Algebr. Represent. Theory 16 (2013), no. 3, 873--894.


\bibitem{AI} T. Aihara and O. Iyama. \emph{Silting mutation in triangulated categories.}  J. London Math. Soc. 85 (2012), no. 2, 633--668.

%\bibitem{Bri} T. Bridgeland.  \emph{Equivalences of triangulated categories and Fourier-Mukai transforms.}  Bull. London Math. Soc. 31 (1999), 25--34.

\bibitem{BPP} N. Broomhead, D. Pauksztello and D. Ploog. \emph{Discrete derived categories I: homomorphisms, autoequivalences and t-structures.} Math. Z. 285 (2017), no. 1-2, 39--89.

\bibitem{Ch} X.-W. Chen.  \emph{Graded self-injective algebras "are'' trivial extensions.}  J. Algebra 322 (2009), no. 7, 2601--2606. 

\bibitem{DIJ} L. Demonet, O. Iyama and G. Jasso. \emph{$\tau$-tilting finite algebras, bricks, and g-vectors.} 
Int. Math. Res. Not. IMRN 2019, no. 3, 852--892. 

\bibitem{Dug} A. Dugas.  \emph{Stable auto-equivalences for local symmetric algebras.}   J. Algebra 449 (2016), 22--49.

\bibitem{DuTr} A. Dugas and B. Trok.  \emph{The stable Picard group of a local dihedral algebra.} Comm. Alg. 46 (2018), no. 6, 2428--2439.

\bibitem{Gra} J. Grant.  \emph{Derived autoequivalences from periodic algebras.}  Proc. Lond. Maht. Soc. (3) 106 (2013), no. 2, 375--409.

\bibitem{Guo} J. Y. Guo.  \emph{Coverings and truncations of graded self-injective algebras.}  J. Algebra 355 (2012), 9--34.

\bibitem{Hap} D. Happel.  \emph{Triangulated Categories in the Representation Theory of Finite Dimensional Algebras.}  London Math. Soc. Lecture Note Ser., vol. 119, Cambridge University Press, 1988.

%\bibitem{FMTAG} D. Huybrechts.  \emph{Fourier-Mukai transforms in algebraic geometry.} Oxford Mathematical Monographs. The Clarendon Press, Oxford University Press, Oxford, 2006.

%\bibitem{HuTh} D. Huybrechts and R. Thomas.  \emph{$\mathbb{P}$-objects and autoequivalences of derived categories.} Math. Res. Lett. 13 (2006), no. 1, 87--98.
\bibitem{KY}  S. Koenig and D. Yang.  \emph{Silting objects, simple-minded collections, t-structures and co-t-structures for finite-dimensional algebras.} Doc. Math. 19 (2014), 403--438. 

%\bibitem{MTDC} J. Rickard.  \emph{Morita theory for derived categories.}  J. London Math. Soc. 39 (1989), no. 2, 436--456.

\bibitem{DCSE} J. Rickard.  \emph{Derived categories and stable equivalence.}  J. Pure Appl. Algebra 61 (1989), no. 3, 303--317.


%\bibitem{EDCSA} J. Rickard.  \emph{Equivalences of derived categories for symmetric algebras.} J. Algebra 257 (2002), no. 2, 460--481.


\bibitem{SeTh} P. Seidel and R. Thomas.  \emph{Braid group actions on derived categories of coherent sheaves.} Duke Math. J. 108 (2001), 37--108.


\end{thebibliography}
 \end{document}